\documentclass[10pt]{article}
\usepackage[margin=1in]{geometry}
\usepackage{amsmath,amssymb,amsthm,mathtools}
\usepackage[all,cmtip]{xy}
\usepackage[hidelinks]{hyperref}
\hypersetup{pdftitle={A counterexample to Enochs conjecture under V=L}, pdfauthor={Chencheng Zhang}}
\xymatrixrowsep{1.5pc}
\allowdisplaybreaks
\newtheorem*{question}{Question}
\newtheorem{theorem}{Theorem}[section]
\newtheorem{lemma}[theorem]{Lemma}
\newtheorem{proposition}[theorem]{Proposition}
\newtheorem{fact}[theorem]{Fact}
\theoremstyle{remark}

\title{A counterexample to Enochs' conjecture under $V=L$}
\author{Chencheng Zhang\thanks{The author was supported by the National Natural Science Foundation of China (No.~12131015).}}
\date{\today}

\begin{document}
\maketitle
\begingroup
\makeatletter
\def\@thefnmark{}
\H@@footnotetext{\textit{2020 Mathematics Subject Classification.} Primary 16D90; Secondary 16E30, 03E45, 03E55.}
\makeatother
\endgroup

\begin{abstract}
Assuming the nonexistence of uncountable measurable cardinals, we construct an $\mathbb F_2$-algebra and a covering class of right modules which is not closed under colimits of countable chains.
This gives a counterexample to Enochs' conjecture under $V=L$.
\end{abstract}

\section{Introduction}

All rings are unital and all modules are right modules.
For a ring $R$, write $\operatorname{Mod}_R$ for its module category.
An isomorphism-closed class $\mathcal X\subseteq\operatorname{Mod}_R$ is \emph{precovering} if every module $M$ admits a map $p:X\longrightarrow M$ with $X\in\mathcal X$, such that every map $X'\longrightarrow M$ with $X'\in\mathcal X$ factors through $p$; such a map $p$ is an $\mathcal X$-precover of $M$.
A precover $p : X \longrightarrow M$ is a \emph{cover} if every $u\in\operatorname{End}_R(X)$ satisfying $pu=p$ is invertible; such a map $p$ is an $\mathcal X$-cover of $M$.
See \cite{EJ} for background on covers and precovers.

Familiar covering classes are closed under filtered colimits, for instance, injective modules over a right Noetherian ring \cite{Enochs81}, flat modules over any ring \cite{BEE}, and Gorenstein flat modules over any ring \cite[Corollary~3.12]{SS}.
Enochs proved that a precovering class closed under filtered colimits is covering \cite{Enochs81}; see also \cite[Corollary~5.2.7]{EJ}.
The converse is known as Enochs' conjecture.
\par\medskip\noindent
\textbf{Conjecture} \emph{Every covering class of modules is closed under filtered colimits.}
\par\medskip

The question is recorded in G\"obel--Trlifaj \cite[\S~5.4, Question~2, p.~130]{GT}.
H\"ugel--\v{S}aroch--Trlifaj \cite[\S~5]{AHT} attribute it to Enochs in the late 1990s; we follow Bazzoni--\v{S}aroch \cite[Introduction]{BS} in using the customary name \emph{Enochs' conjecture}.

Several cases of the conjecture are known.
Bazzoni--\v{S}aroch proved it for classes $\operatorname{Filt}(\mathcal S)$, consisting of modules filtered by a set $\mathcal S$ of $<\aleph_n$-presented modules for some fixed $n<{\aleph_0}$; under $V=L$, they proved it for every set $\mathcal S$ \cite[Corollaries~2.11 and~3.5]{BS}.
Iacob proved that the class of Gorenstein injective modules is covering if and only if it is closed under filtered colimits \cite[Theorem~3]{Iacob}.

This article gives a conditional counterexample.
\par\medskip\noindent
\textbf{Theorem} (see Theorem~\ref{thm:construction}). \emph{If no uncountable measurable cardinal exists, then there is an $\mathbb F_2$-algebra $\Lambda$ and a covering class in $\operatorname{Mod}_\Lambda$ which is not closed under colimits of countable chains.}
\par\medskip

The hypothesis that no uncountable measurable cardinal exists follows from $V=L$, an axiom that is independent of ZFC, provided ZFC is consistent; see \cite[\S~5]{Goldberg} for details.
\par\medskip\noindent
\textbf{Corollary.} \emph{Under $V=L$, there is a counterexample to Enochs' conjecture.}
\par\medskip
\begin{question}
\leavevmode
Exactly one of the following two questions has an affirmative answer; we do not know which one.
\begin{enumerate}
\item Can a counterexample to Enochs' conjecture be constructed in ZFC?
\item Is Enochs' conjecture independent of ZFC, provided ZFC is consistent?
\end{enumerate}
\end{question}

\section{Notation and preliminaries}
\paragraph{Notations of modules.}
Note that $\operatorname{Ext}_R^1(X,N)$ is the group of extensions $0\longrightarrow N\longrightarrow E\longrightarrow X\longrightarrow0$ modulo equivalence, with Baer sum as addition.
For $f:N\longrightarrow N'$ and $g:X'\longrightarrow X$, write $f_*:\operatorname{Ext}_R^1(X,N)\longrightarrow\operatorname{Ext}_R^1(X,N')$ and $g^*:\operatorname{Ext}_R^1(X,N)\longrightarrow\operatorname{Ext}_R^1(X',N)$ for pushout and pullback, respectively.
Direct products and direct sums are denoted by $M^I=\prod_{i\in I}M$ and $M^{(I)}=\bigoplus_{i\in I}M$, where $M^{(I)}$ is the submodule of finite-support tuples in $M^I$.
For a system of modules indexed by $J$, a filtered colimit $\varinjlim_JM$ is a colimit in which every finite diagram in $J$ admits a cocone.
Our example uses the chain $M_0\longrightarrow M_1\longrightarrow\cdots$.

\begin{fact}
For $X$ and $(N_i)_{i\in I}$ in $\operatorname{Mod}_R$, the projections $\pi_i:\prod_{j\in I}N_j\longrightarrow N_i$ induce a natural isomorphism
\begin{equation}\label{eq:extproduct}
 \operatorname{Ext}_R^1(X,\prod_{i\in I}N_i)\overset{\sim}{\longrightarrow}\prod_{i\in I}\operatorname{Ext}_R^1(X,N_i),\qquad \eta\longmapsto\bigl((\pi_i)_*\eta\bigr)_{i\in I}.
\end{equation}
\end{fact}

\paragraph{Ultrafilters.}
For a nonempty set $I$, write $\mathcal P(I)$ for its power set.
A \emph{filter} on $I$ is a family $\mathcal F\subseteq\mathcal P(I)$ containing $I$ but not $\varnothing$, closed under finite intersections and passage to supersets.
An \emph{ultrafilter} $\mathcal U$ is a filter maximal with respect to inclusion.
\begin{fact}\label{fact:ultrafilter}
  A filter $\mathcal F$ on $I$ is an ultrafilter if and only if for every $A\subseteq I$, either $A$ or $I\setminus A$ belongs to $\mathcal F$.
\end{fact}
An ultrafilter $\mathcal U$ is \emph{principal} if $\mathcal U=\{A\subseteq I\mid i\in A\}$ for some $i\in I$, and \emph{nonprincipal} otherwise.
Note that principal ultrafilters are ultrafilters.
For an infinite cardinal $\kappa$, an ultrafilter is \emph{$\kappa$-complete} if it is closed under intersections of fewer than $\kappa$ members.
Let ${\aleph_0}$ be the countable cardinal. 
Then all filteres are ${\aleph_0}$-complete.

\paragraph{Constructible sets.}
For a set $A$, let $\operatorname{Def}(A)$ consist of the subsets $\{x\in A\mid (A,\in)\models\varphi(x,a_1,\ldots,a_n)\}$, where $\varphi$ is a first-order formula and $a_1,\ldots,a_n\in A$.
Define the cumulative and constructible hierarchies:
\begin{itemize}
  \item $V_0=\varnothing$, $V_{\alpha+1}=\mathcal P(V_\alpha)$ the power set, and $V_\lambda=\bigcup_{\alpha<\lambda}V_\alpha$ for limit ordinals $\lambda$.
  \item $L_0=\varnothing$, $L_{\alpha+1}=\operatorname{Def}(L_\alpha)$ the set of definable subsets, and $L_\lambda=\bigcup_{\alpha<\lambda}L_\alpha$ for limit ordinals $\lambda$.
\end{itemize}
Put $V=\bigcup_{\alpha\in\mathrm{Ord}}V_\alpha$ and $L=\bigcup_{\alpha\in\mathrm{Ord}}L_\alpha$.
Thus $V=L$ says that every set is constructible.
G\"odel proved that $L$ is an inner model of ZFC satisfying $V=L$.
Cohen's forcing gives the relative consistency of $V\ne L$.
Together these imply the independence \cite[\S~5]{Goldberg}, stated in the introduction.

\paragraph{Measurable cardinals.}
An uncountable cardinal $\kappa$ is \emph{measurable} if it carries a $\kappa$-complete nonprincipal ultrafilter.
Scott proved that $V=L$ implies the nonexistence of uncountable measurable cardinals \cite{Scott}.

  \begin{lemma}\label{lem:complete-ultrafilters}
  If no uncountable measurable cardinal exists, then every $\aleph_1$-complete ultrafilter is principal.
  \end{lemma}

  \begin{proof}
  Suppose that $\mathcal U$ is an $\aleph_1$-complete nonprincipal ultrafilter on a set $I$.
  Let $\kappa$ be the least cardinality of a partition of $I$ for which none of the parts belong to $\mathcal U$.
  Such a partition exists because every singleton lies outside $\mathcal U$.
  By definition of $\aleph_1$-completeness, one has $\kappa>\aleph_0$.

  \emph{We claim that a union of fewer than $\kappa$ sets outside $\mathcal U$ also lies outside $\mathcal U$:
  take $\lambda<\kappa$ with $B_\alpha \in \mathcal{P}(I) \setminus \mathcal U$, 
  and it suffices to show that $\bigcup_{\alpha<\lambda}B_\alpha\notin\mathcal U$.}
  Otherwise, put
  \[
  C_\alpha=B_\alpha\setminus\bigcup_{\beta<\alpha}B_\beta,
  \qquad
  D=I\setminus\bigcup_{\alpha<\lambda}B_\alpha.
  \]
  Then $I=D\sqcup\bigsqcup_{\alpha<\lambda}C_\alpha$ is a partition of $I$ with $D\notin\mathcal U$ and $C_\alpha\notin\mathcal U$.
  Since $|\lambda|+1<\kappa$, this gives fewer than $\kappa$ parts, contradicting minimality.
  This proves the claim.

  Fix such partition $(A_\alpha)_{\alpha<\kappa}$ witnessing the definition of $\kappa$, and put $\mathcal V=\left\{S\subseteq\kappa\mid\bigcup_{\alpha\in S}A_\alpha\in\mathcal U\right\}$.
  By construction $\mathcal V$ is a filter, and for each $S \subseteq \kappa$ either $S$ or $\kappa \setminus S$ belongs to $\mathcal V$.  
  By Fact~\ref{fact:ultrafilter}, $\mathcal V$ is an ultrafilter on $\kappa$.
  It is nonprincipal, since $A_\alpha\notin\mathcal U$ for every $\alpha<\kappa$.
  By the claim above, $\mathcal V$ is $\kappa$-complete.
  Thus $\kappa$ is an uncountable measurable cardinal, a contradiction.
  This completes the proof.
  \end{proof}

\section{Main construction}

\paragraph{Setup.}

Put $k=\mathbb F_2$ and $K=k^{({\aleph_0})}$, viewed as finite-support row vectors.
Put
\[
 R=\left\{(r_{ij})\in k^{{\aleph_0}\times{\aleph_0}}\mid \{j\in{\aleph_0}\mid r_{ij}\ne0\}\text{ is finite for every }i\in{\aleph_0}\right\},\qquad \Lambda=\begin{psmallmatrix}R&0\\R&R\end{psmallmatrix}.
\]
Thus $R$ is the ring of row-finite matrices over $k$, and $K$ is a right $R$-module.

\begin{fact}
Let $\operatorname{Mod}_R^{\to}$ denote the morphism category of $\operatorname{Mod}_R$, whose objects are module maps and whose morphisms are commutative squares.
There is an equivalence of categories $\operatorname{Mod}_R^{\to}\xlongrightarrow{\sim}\operatorname{Mod}_\Lambda$, where an arrow $a:X\longrightarrow Y$ is sent to the $\Lambda$-module $Y\oplus X$ with action $(y,x)\begin{psmallmatrix}r&0\\s&t\end{psmallmatrix}=(yr+a(x)s,xt)$.
\end{fact}

Put $\mathcal C=\{a:X\longrightarrow Y\mid a^*:\operatorname{Ext}_R^1(Y,K)\longrightarrow\operatorname{Ext}_R^1(X,K) \text{ is a zero map}\}$; the main theorem follows:
\begin{theorem}\label{thm:construction}
If no uncountable measurable cardinal exists, then the class $\mathcal C$ is covering in $\operatorname{Mod}_\Lambda$ but $\mathcal{C}$ is not closed under colimits of countable chains.
\end{theorem}

\paragraph{The key proposition.}

Let $(v_n)_{n<{\aleph_0}}$ be the basis of $K$.
For each $m$ and $n$ in ${\aleph_0}$, let $E_{mn} \in R$ be the matrix with its only nonzero entry $1$ in position $(m,n)$.
Put $e_n=E_{nn}$.

\begin{lemma}\label{lem:transitive}
  For each nonzero $x,y \in K$, there exists a unit $u \in R$ such that $xu = y$.
\end{lemma}

\begin{proof}
  Take $N$ such that $x_n = y_n = 0$ for all $n \ge N$.
  Then one can extend $x$ and $y$ to bases of the finite-dimensional subspace $K_N = \bigoplus_{n<N} k v_n$, respectively.
  Then one can take $u = \begin{psmallmatrix}u_0 & 0 \\ 0 & 1\end{psmallmatrix}$ as a diagonal matrix with $u_0$ an automorphism of $K_N$ sending $x$ to $y$.
  Thus $u$ is a unit in $R$ sending $x$ to $y$.
\end{proof}

\begin{lemma}\label{lem:scalars}
  The map $k \longrightarrow \operatorname{End}_R(K)$ sending each $c\in k$ to its scalar is an isomorphism of rings.
\end{lemma}

\begin{proof}
  The map is clearly injective.
  Now fix an arbitrary $f \in \mathrm{End}_R(K)$.
  For each $n < {\aleph_0}$, one has $f(v_n) = f(v_n e_n) = f(v_n) e_n$.
  Thus $f(v_n)$ is a scalar multiple of $v_n$, denoted by $c_n v_n$ for some $c_n \in k$.
  For any $m \neq n$, since $v_m E_{mn}=v_n$, one has $c_nv_n=f(v_n)=f(v_mE_{mn})=f(v_m)E_{mn}=c_mv_n$.
  Hence $c_n=c_m$ for all $m,n<{\aleph_0}$, and $f$ is a scalar multiplication.
  This completes the proof.
\end{proof}

\begin{proposition}[The key proposition]\label{prop:finite-coordinates}
If no uncountable measurable cardinal exists, then every $R$-module homomorphism $K^I\longrightarrow K$ is uniquely of the form $\sum_{i\in I}c_i\pi_i$, where $(c_i)_{i\in I}\in k^{(I)}$.
\end{proposition}
\begin{proof}
\emph{We claim that any homomorphism $h:\prod_{n<{\aleph_0}}M_n\longrightarrow K$ of $R$-modules killing $\bigoplus_{n<{\aleph_0}}M_n$ is zero.}
Suppose otherwise that $h \neq 0$.
Then there is $h(x)$ with $r$th coordinate nonzero.
By Lemma~\ref{lem:transitive}, assume $h(x) = v_0$.
Write $x=(x_j)_{j<{\aleph_0}}$ and set $z=(z_j)_{j<{\aleph_0}}$ with $z_j=\sum_{n\le j}x_jE_{0n}$.
For each $m < {\aleph_0}$, one has $z_je_m=x_jE_{0m}$ if $j\ge m$.
Thus $ze_m-xE_{0m}$ has finite support and is killed by $h$.
Hence $h(z)e_m=h(x)E_{0m}=v_m$ for each $m < {\aleph_0}$.
This is impossible, since $h(z)\in K$ has finite support.
Therefore the claim is proved.

Now fix a nonzero $R$-module homomorphism $h:K^I\longrightarrow K$.
For $J\subseteq I$, put $\mathcal N=\{J\subseteq I\mid h|_{K^J}=0\}$. 
By the claim above, $\mathcal{N}$ is closed under countable disjoint unions.
Since $\mathcal{N}$ is also closed under subsets, it is closed under countable unions.
Since $\mathcal{P}(I)$ is a Boolean algebra with $\mathcal{N}$ an ideal, write the quotient algebra 
\[ 
\mathcal{P}(I) \longrightarrow \mathcal B=\mathcal P(I)/\mathcal N, \qquad S \longmapsto [S].
\]

\emph{We claim that $\mathcal B$ is finite.}
Suppose otherwise.
One can take a countable family of pairwise disjoint nonzero subsets $(J_n)_{n < {\aleph_0}}$ of $I$ such that $h|_{K^{J_n}}\ne0$ for every $n<{\aleph_0}$.
For each $n$, choose $x_n\in K^{J_n}$ with $h(x_n)\ne0$.
By Lemma~\ref{lem:transitive}, choose a unit $u_n\in R$ such that $h(x_n)u_n=v_n$, and put $y_n=x_nu_ne_n$.
Then $h(y_n)=v_n$ and $y_ne_n=y_n$.
Since the $J_n$'s are pairwise disjoint, one takes $y\in K^I$ with $y|_{J_n}=y_n$ and sets all remaining coordinates to zero.
Each $y_n$ has only its $n$th $K$-coordinate nonzero, so $ye_n=y_n$.
Consequently,
\[
 h(y)e_n=h(ye_n)=h(y_n)=v_n\qquad \text{for each $n<{\aleph_0}$}.
\]
Thus $h(y)$ has infinitely many nonzero coordinates, a contradiction.
This proves the claim.

Since $\mathcal B$ is finite, choose a complete set of pairwise orthogonal primitive idempotents $b_1,\ldots,b_m$.
Fix arbitrary $1 \leq i \leq m$.
Put $\mathcal U_i=\{J\subseteq I\mid b_i\le[J]\}$ which is an ultrafilter, 
and fix $C_i\subseteq I$ with $[C_i]=b_i$.
For $J_n\in\mathcal U_i$ with $n<\omega$, one has $C_i\setminus J_n\in\mathcal N$.  
Since $\mathcal N$ is closed under countable unions, one has
\[
  C_i\setminus\bigcap_{n<\omega}J_n =\bigcup_{n<\omega}(C_i\setminus J_n)\in\mathcal N.
\]
Thus $\mathcal U_i$ is countably complete, hence principal by Lemma~\ref{lem:complete-ultrafilters}.
Hence for each $1 \leq i \leq m$, there is a unique $p_i\in I$ such that $\mathcal U_i=\{J\subseteq I\mid p_i\in J\}$.
Put $F=\{p_1,\ldots,p_m\}$.
Since $F\in\mathcal U_i$ for every $i$, one has $b_i\le[F]$ for every $i$, and therefore $[F]=1$.
Hence $I\setminus F\in\mathcal N$, so $h$ factors through $K^F$.
By Lemma~\ref{lem:scalars}, this gives $h=\sum_{p\in F}c_p\pi_p$ with $c_p\in k$.

Restriction to each coordinate copy of $K$ gives uniqueness.
\end{proof}

\paragraph{Verification of the covering class.}

Assume that no uncountable measurable cardinal exists.
Fix an object in $\operatorname{Mod}_R^\to$, say $a:X\longrightarrow Y$.
Put
\[
S_a=\operatorname{im}\bigl(a^*:\operatorname{Ext}_R^1(Y,K)\longrightarrow\operatorname{Ext}_R^1(X,K)\bigr).
\]
Choose a $k$-basis $(\xi_i)_{i\in I}$ of $S_a$.
Fix $\zeta_i\in\operatorname{Ext}_R^1(Y,K)$ with $a_K^*(\zeta_i)=\xi_i$.
By \eqref{eq:extproduct}, choose $\theta\in\operatorname{Ext}_R^1(Y,K^I)$ such that $(\pi_i)_*\theta=\zeta_i$ for every $i\in I$.
Put $\eta=a^*\theta$, represented by the pullback sequence
\[
\xymatrix@R=12pt{
 \eta:&0\ar[r]&K^I\ar[r]^j\ar@{=}[d]&E\ar[r]^p\ar[d]&X\ar[r]\ar[d]^a&0\\
 \theta:&0\ar[r]&K^I\ar[r]&E'\ar[r]&Y\ar[r]&0.
}
\]
Naturality gives $(\pi_i)_*\eta=a^*((\pi_i)_*\theta)=a^*\zeta_i=\xi_i$ for every $i\in I$.
By the key proposition (Proposition~\ref{prop:finite-coordinates}) the connecting map is an isomorphism
\begin{equation}\label{eq:connecting}
 \partial_\eta:\operatorname{Hom}_R(K^I,K)\xlongrightarrow{\sim}S_a, \qquad \pi_i\longmapsto\xi_i.
\end{equation}

\begin{proposition}\label{prop:covers}
The square $(p,1_Y):(E\overset{ap}{\longrightarrow}Y)\longrightarrow(X\overset{a}{\longrightarrow}Y)$ is a $\mathcal C$-cover.
\end{proposition}
\begin{proof}
  We show that $(p,1_Y)$ is a $\mathcal C$-precover.
  Since $p^*S_a = \operatorname{im} p^\ast \partial_\eta =0$, $ap\in\mathcal C$.
For $c:U\longrightarrow V$ in $\mathcal C$ and a square $(u,v):c\longrightarrow a$, one has $u^*a^*=c^*v^*=0$.
For each $i\in I$, binaturality of pushout and pullback gives
  \[
  (\pi_i)_*(u^*\eta)  =u^*((\pi_i)_*\eta)  =u^*\xi_i  =u^*a^*\zeta_i  =c^*v^*\zeta_i  =0.
  \]
  The last equality is by $c\in\mathcal C$.
  All coordinates of $u^*\eta$ vanish, hence $u^*\eta=0$ by \eqref{eq:extproduct}.
  Thus $u$ lifts through $p$:
  \[
  \xymatrix@R=12pt{
   U\ar[d]_c\ar@{.>}[r]_{\widetilde u}\ar@/^1pc/[rr]^u
   &E\ar[d]^{ap}\ar[r]_p
   &X\ar[d]^a\\
   V\ar[r]_v&Y\ar@{=}[r]&Y.
  }
  \]
  Hence $\mathcal{C}$ is precovering.

  To see that $\mathcal{C}$ is covering, take $(s, t) : (E\overset{ap}{\longrightarrow}Y)\longrightarrow (E\overset{ap}{\longrightarrow}Y)$ with $(p,1_Y)(s,t)=(p,1_Y)$.
  Then $t = 1_Y$ and $s$ is an endomorphism of $E$ with $ps=p$, yielding a commutative diagram of exact sequences
\[
\xymatrix@R=12pt{
 0\ar[r]&K^I\ar[r]^j\ar[d]_\alpha&E\ar[r]^p\ar[d]_s&X\ar[r]\ar@{=}[d]&0\\
 0\ar[r]&K^I\ar[r]_j&E\ar[r]_p&X\ar[r]&0.
}
\]
Thus $\alpha_*\eta=\eta$.
By \eqref{eq:connecting}, $\pi_i\alpha=\pi_i$ for every $i$, so $\alpha=1$.
Hence $s$ is invertible, and $(s,t)$ is invertible in $\operatorname{Mod}_R^\to$.
It follows that $(p,1_Y)$ is a $\mathcal C$-cover.
\end{proof}

\paragraph{Nonclosedness under filtered colimits.}

Put $J_m=\bigoplus_{n<m}e_nR$, $J=\bigoplus_{n<{\aleph_0}}e_nR$, $M_m=R/J_m$ and $M=R/J$.
Since $J_m=(e_0+\cdots+e_{m-1})R$ is a direct summand of $R$, each $M_m$ is cyclic projective.
The quotient maps give $\varinjlim_mM_m=M$.
Applying $\operatorname{Hom}_R(-,K)$ to $0\longrightarrow J\longrightarrow R\longrightarrow M\longrightarrow0$ gives
\[
\xymatrix@R=12pt{
 \operatorname{Hom}_R(R,K)\ar[r]\ar[d]_{\cong}&\operatorname{Hom}_R(J,K)\ar[r]\ar[d]^{\cong} &\operatorname{Ext}_R^1(M,K)\ar[r]&0\\
 K\ar[r]&\prod_{n<{\aleph_0}}Ke_n.&&
}
\]
Under these identifications, the restriction map becomes
\[
K\longrightarrow\prod_{n<\aleph_0}Ke_n,\qquad x\longmapsto(xe_n)_{n<\aleph_0}.
\]
Since $Ke_n=kv_n$, this is the natural inclusion $k^{(\aleph_0)}\longrightarrow k^{\aleph_0}$.
Hence $\operatorname{Ext}_R^1(M,K)\cong k^{\aleph_0}/k^{(\aleph_0)}\ne0$, where the constant sequence $(1,1,\ldots)$ represents a nonzero class.
Hence the chain
\[
\xymatrix@R=12pt{
 M_0\ar[r]\ar[d]_{1}&M_1\ar[r]\ar[d]_{1}&\cdots\ar[r]&M\ar[d]^{1}\\
 M_0\ar[r]&M_1\ar[r]&\cdots\ar[r]&M
}
\]
has $1_{M_m}\in\mathcal C$ at every stage but $1_M\notin\mathcal C$ at its colimit.
This proves Theorem~\ref{thm:construction}.




\par\bigskip
\begingroup
\small
\noindent\textsc{School of Mathematical Sciences, Shanghai Jiao Tong University, Shanghai 200240, P. R. China}\par
\noindent\textit{Email address:} \href{mailto:zhangchencheng@sjtu.edu.cn}{\texttt{zhangchencheng@sjtu.edu.cn}}
\par\endgroup


\begin{thebibliography}{GGK09}
\raggedright

\bibitem[AHT18]{AHT}
L.~Angeleri H\"ugel, J.~\v{S}aroch and J.~Trlifaj, \emph{Approximations and Mittag-Leffler conditions---the applications}, Israel J. Math. \textbf{226} (2018), 757--780.
\url{https://arxiv.org/abs/1612.01140}.

\bibitem[B{\v S}23]{BS}
S.~Bazzoni and J.~{\v S}aroch, \emph{Enochs conjecture for cotorsion pairs and more}, arXiv:2303.08471v3 (2023).
\url{https://arxiv.org/abs/2303.08471}.

\bibitem[BEE01]{BEE}
L.~Bican, R.~El Bashir and E.~E.~Enochs, \emph{All modules have flat covers}, Bull. London Math. Soc. \textbf{33} (2001), no.~4, 385--390.

\bibitem[Eno81]{Enochs81}
E.~E.~Enochs, \emph{Injective and flat covers, envelopes and resolvents}, Israel J. Math. \textbf{39} (1981), 189--209.

\bibitem[EJ11]{EJ}
E.~E.~Enochs and O.~M.~G.~Jenda, \emph{Relative Homological Algebra, Volume 1}, second revised and extended edition, De Gruyter Expositions in Mathematics 30, De Gruyter, Berlin/Boston, 2011.

\bibitem[GT12]{GT}
R.~G\"obel and J.~Trlifaj, \emph{Approximations and Endomorphism Algebras of Modules, Volume 1: Approximations}, second revised and extended edition, De Gruyter Expositions in Mathematics 41, De Gruyter, Berlin/Boston, 2012.

\bibitem[Gol24]{Goldberg}
G.~Goldberg, \emph{Fudan notes on UA}, lecture notes (2024), Section~5.
\url{https://math.berkeley.edu/~goldberg/Slides/FudanNotes.pdf}.

\bibitem[Iac24]{Iacob}
A.~Iacob, \emph{The class of Gorenstein injective modules is covering if and only if it is closed under direct limits}, arXiv:2403.02493v2 (2024).
\url{https://arxiv.org/abs/2403.02493}.

\bibitem[Sco61]{Scott}
D.~Scott, \emph{Measurable cardinals and constructible sets}, Bull. Acad. Polon. Sci. S\'{e}r. Sci. Math. Astronom. Phys. \textbf{9} (1961), 521--524.

\bibitem[\v{S}\v{S}20]{SS}
J.~\v{S}aroch and J.~\v{S}\v{t}ov\'{i}\v{c}ek, \emph{Singular compactness and definability for $\Sigma$-cotorsion and Gorenstein modules}, Selecta Math. (N.S.) \textbf{26} (2020), article~23.
\url{https://arxiv.org/abs/1804.09080}.

\end{thebibliography}
\end{document}